\documentclass[12pt]{article}
\usepackage{amsmath,amsfonts,amssymb,amsthm,latexsym}
\theoremstyle{plain}
\newtheorem{theorem}{Theorem}
\newtheorem{corollary}[theorem]{Corollary}
\newtheorem{lemma}[theorem]{Lemma}

\theoremstyle{definition}
\newtheorem{definition}[theorem]{Definition}

\theoremstyle{remark}

\begin{document}
\begin{center}
\textbf{Relatively Prime Sets, Divisor Sums, and Partial Sums}
\end{center}
\begin{center}
Prapanpong Pongsriiam\\
Department of Mathematics, Faculty of Science, Silpakorn University, Ratchamankanai Rd, Nakornpathom, Thailand, 73000.\\
 Email: prapanpong@gmail.com
\end{center}
\begin{center}
\textbf{Abstract}
\end{center}
For a nonempty finite set $A$ of positive integers, let $\gcd\left(A\right)$ denote the greatest common divisor of the elements of $A$. Let $f\left(n\right)$ and  $\Phi\left(n\right)$ denote, respectively, the number of subsets $A$ of $\left\{1, 2, \ldots, n\right\}$ such that $\gcd\left(A\right) = 1$ and the number of subsets $A$ of $\left\{1, 2, \ldots, n\right\}$ such that $\gcd\left(A\cup\left\{n\right\}\right) =1$. Let $D\left(n\right)$ be the divisor sum of $f\left(n\right)$. In this article, we obtain partial sums of $f\left(n\right)$, $\Phi\left(n\right)$ and $D\left(n\right)$. We also obtain a combinatorial interpretation and a congruence property of $D\left(n\right)$. We give open questions concerning $\Phi\left(n\right)$ and $D\left(n\right)$ at the end of this article.

\section{Introduction}
Unless stated otherwise, we let $d, k, n, N$ be positive integers, $A$ a nonempty finite set of positive integers, $\gcd\left(A\right)$ the greatest common divisor of the elements of $A$, $\left\lfloor x \right\rfloor$ the greatest integer less than or equal to $x$, and $\mu$ the M$\ddot{\text{o}}$bius function. 

	$A$ is said to be \textit{relatively prime} if $\gcd \left(A\right)=1$ and is said to be \textit{relatively prime to $n$} if $\gcd\left(A\cup\left\{n\right\}\right) = 1$. Let $f\left(n\right)$ and $\Phi\left(n\right)$ denote, respectively, the number of relatively prime subsets of $\left\{1, 2, \ldots, n\right\}$, and the number of nonempty subsets of $\left\{1, 2, \ldots, n\right\}$ relatively prime to $n$. In addition, we let $D\left(n\right)=\sum_{d\mid n}f\left(d\right)$ be the divisor sum of $f\left(n\right)$. The first 15 values of $f\left(n\right)$, $\Phi\left(n\right)$, and $D\left(n\right)$ are given in the table below.
	
\begin{table}[ht]
\begin{center}
\begin{tabular}{|c|c|c|c|c|}
\hline
$n$&$f\left(n\right)$&$\Phi\left(n\right)$&$D\left(n\right)$&$2^n$\\\hline
1&1&1&1&2\\
2&2&2&3&4\\
3&5&6&6&8\\
4&11&12&14&16\\
5&26&30&27&32\\
6&53&54&61&64\\
7&116&126&117&128\\
8&236&240&250&256\\
9&488&504&494&512\\
10&983&990&1012&1024\\
11&2006&2046&2007&2048\\
12&4016&4020&4088&4096\\
13&8111&8190&8112&8192\\
14&16238&16254&16357&16384\\
15&32603&32730&32635&32768\\
\hline
\end{tabular}
\end{center}
\caption{The first 15 values of $f\left(n\right)$, $\Phi\left(n\right)$, and $D\left(n\right)$.}\label{tablerela1}
\end{table}
\newpage
The purpose of this article is to obtain partial sums associated with $f\left(n\right)$, $\Phi\left(n\right)$, and $D\left(n\right)$ and use them to explain some phenomena appear in Table \ref{tablerela1}. We will also obtain a combinatorial interpretation and a congruence property of $D\left(n\right)$. An open problem arising from an observation on the values of $\Phi\left(n\right)$ and $D\left(n\right)$ is also given. By way of example, the formulas of the partial sums of $f\left(n\right)$, $\Phi\left(n\right)$, and $D\left(n\right)$  lead to the following results: (see Corollary \ref{1to6} for the proof),
\begin{equation}\label{relintrolimsupeq1}
\limsup_{N\rightarrow\infty} \frac{\left|\sum_{n\leq N}f\left(n\right)-2^{N+1}\right|}{2^{\frac{N}{2}}} = 3
\end{equation}
\begin{equation}\label{relintrolimsupeq2}
\liminf_{N\rightarrow\infty} \frac{\left|\sum_{n\leq N}f\left(n\right)-2^{N+1}\right|}{2^{\frac{N}{2}}} = 2\sqrt 2
\end{equation}
\begin{equation}\label{relintrolimsupeq3}
\limsup_{N\rightarrow\infty} \frac{\left|\sum_{n\leq N}\Phi\left(n\right)-2^{N+1}\right|}{2^{\frac{N}{2}}} = 2
\end{equation}
\begin{equation}\label{relintrolimsupeq4}
\liminf_{N\rightarrow\infty} \frac{\left|\sum_{n\leq N}\Phi\left(n\right)-2^{N+1}\right|}{2^{\frac{N}{2}}} = \sqrt 2
\end{equation}
\begin{equation}\label{relintrolimsupeq5}
\limsup_{N\rightarrow\infty} \frac{\left|\sum_{n\leq N}D\left(n\right)-2^{N+1}\right|}{2^{\frac{N}{2}}} = \sqrt 2
\end{equation}
\begin{equation}\label{relintrolimsupeq6}
\liminf_{N\rightarrow\infty} \frac{\left|\sum_{n\leq N}D\left(n\right)-2^{N+1}\right|}{2^{\frac{N}{2}}} = 1
\end{equation}

\section{Preliminaries and Lemmas}
Let $E\left(n\right) = \sum_{d\mid n}\Phi\left(d\right)$ be the divisor sum of $\Phi\left(n\right)$. By the definition of $f\left(n\right)$, $\Phi\left(n\right)$, $D\left(n\right)$ and $E\left(n\right)$ and the results obtained by Nathanson \cite{Na}, the following holds
\begin{align}\label{relpreandlemeq1}
f\left(n\right)\leq \min\{\Phi\left(n\right),D\left(n\right)\}\leq \max\{\Phi\left(n\right),D\left(n\right)\}\leq E\left(n\right) =2^{n}-1\leq 2^n
\end{align}
Moreover $f\left(n\right)$ is asymptotic to $2^n$. So all functions above are asymptotic to $2^n$. In other words,
\begin{equation}\label{sim}
\lim_{n\rightarrow \infty}\frac{f\left(n\right)}{2^n} = \lim_{n\rightarrow \infty}\frac{\Phi\left(n\right)}{2^n} = \lim_{n\rightarrow \infty}\frac{D\left(n\right)}{2^n} = \lim_{n\rightarrow \infty}\frac{E\left(n\right)}{2^n} = 1
\end{equation} 
So basically, $f\left(n\right)$, $\Phi\left(n\right)$, $D\left(n\right)$, and $E\left(n\right)$ are very closed to $2^n$ as $n\rightarrow \infty$. Which one is closer? We see from (\ref{relpreandlemeq1}) that $\Phi\left(n\right)$ and $D\left(n\right)$ are closer to $2^n$ than $f\left(n\right)$. In addition, $E\left(n\right)$ is closer to $2^n$ than $\Phi\left(n\right)$ and $D\left(n\right)$. But it is not clear (see Table \ref{tablerela1}) which of $\Phi\left(n\right)$ or $D\left(n\right)$ is closer to $2^n$. One way to answer this, at least on average, is to calculate the partial sums $\sum_{n\leq N}\Phi\left(n\right)$ and $\sum_{n\leq N}D\left(n\right)$ and compare them with the expected value $\sum_{n\leq N}2^n = 2^{N+1}-2$. To accomplish this task, we will use the following results.
\begin{lemma}\label{NaAK}
(Nathanson, \cite{Na}) The following holds:
\begin{itemize}
\item [(i)] $\displaystyle f\left(n\right) = \sum_{d\leq n}\mu\left(d\right)\left(2^{\left\lfloor \frac{n}{d}\right\rfloor}-1\right)$ for every $n\geq 1$
\item [(ii)] $\displaystyle \Phi\left(n\right) = \sum_{d\mid n}\mu\left(d\right)\left(2^{ \frac{n}{d}}-1\right)$ for every $n\geq 1$
\end{itemize}
\end{lemma}
\begin{lemma}\label{AKnew}
(Ayad and Kihel \cite{AK}) The following holds:
\begin{itemize}
\item [(i)] $\Phi(n+1) = 2(f(n+1)-f\left(n\right))$ for every $n\geq 1$
\item [(ii)] $\Phi\left(n\right) \equiv 0\pmod 3$ for every $n\geq 3$
\end{itemize}
\end{lemma}

\textbf{Notes}
\begin{itemize}
\item [1)] The functions $f\left(n\right)$ and $\Phi\left(n\right)$ are introduced by Nathason \cite{Na} and generalized by many authors \cite{AK1,AK2,Po,Sh1,Sh2,To}. We refer the reader to Pongsriiam's article (\cite{Po} or \cite{Po1}) for a unified approach and the shortest calculation of the formulas for $f\left(n\right)$, $\Phi\left(n\right)$ and their generalizations. Other related results can be found, for example, in the article of El Bachraoui \cite{Ba5}, El Bachraoui and Salim \cite{BS}, and Tang \cite{Ta}.
\item [2)] The sequences $f\left(n\right)$ and $\Phi\left(n\right)$ are, respectively, Sloane's sequence A085945 and A038199. Note also that A038199 and A027375 coincide for all $n\geq 2$ (see the comments at the end of this article).
\end{itemize}

\section{Partial Sums and Limits}
In this section, we compute the partial sums of $f\left(n\right)$, $\Phi\left(n\right)$, and $D\left(n\right)$. Then we show how to obtain the limits shown in (\ref{relintrolimsupeq1}) to (\ref{relintrolimsupeq6}). Throughout, for a real value function $f$ and a positive function $g$, $f = O\left(g\right)$ or $f\ll g$ means that there exists a positive constant $c$ such that $\left|f\left(x\right)\right|\leq cg\left(x\right)$ for all large numbers $x$.

\begin{theorem}\label{relpartialthm2}
The following holds uniformly for $N\geq 1$.
\begin{itemize}
\item[(i)] $ \displaystyle \sum_{n\leq N}f\left(n\right) = \sum_{d\leq N}d\mu\left(d\right)2^{\left\lfloor \frac{N}{d}\right\rfloor}+\sum_{d\leq N}\mu\left(d\right)2^{\left\lfloor \frac{N}{d}\right\rfloor}$$\left(N-d\left\lfloor \frac{N}{d}\right\rfloor+1\right)+O\left(N^2\right)$\\
 $=$ $2^{N+1}-2^{\left\lfloor \frac{N}{2}\right\rfloor}\left(N-2\left\lfloor \frac{N}{2}\right\rfloor+3\right)-2^{\left\lfloor \frac{N}{3}\right\rfloor}\left(N-3\left\lfloor \frac{N}{3}\right\rfloor+4\right)+O\left(2^{\frac{N}{5}}\right)$.
\item[(ii)] $\displaystyle \sum_{n\leq N}\Phi\left(n\right)= 2f(N)-1 = 2^{N+1}-2\cdot2^{\left\lfloor \frac{N}{2}\right\rfloor}-2\cdot2^{\left\lfloor \frac{N}{3}\right\rfloor}+O\left(2^{\frac{N}{5}}\right)$.
\item[(iii)] $\displaystyle \sum_{n\leq N}D\left(n\right) = 2^{N+1}-2^{\left\lfloor \frac{N}{2}\right\rfloor}$$\left(N-2\left\lfloor \frac{N}{2}\right\rfloor+1\right)+O\left(N2^{\frac{N}{3}}\right)$.
\end{itemize}
\end{theorem}

\begin{proof}
Let $N$ be a large positive integer. Then 
\begin{align*}
\sum_{n\leq N}f\left(n\right) &= \sum_{n\leq N}\sum_{d\leq n}\mu\left(d\right)\left(2^{\left\lfloor \frac{n}{d}\right\rfloor}-1\right)\\
&=\sum_{n\leq N}\sum_{d\leq n}\mu\left(d\right)2^{\left\lfloor \frac{n}{d}\right\rfloor}+O\left(N^2\right).
\end{align*}
Changing the order of summation, we obtain
\begin{equation}\label{parsumandlimstar}
\sum_{n\leq N}f\left(n\right) =\sum_{d\leq N}\mu\left(d\right)\sum_{d\leq n\leq N}2^{\left\lfloor \frac{n}{d}\right\rfloor}+O(N^2)
\end{equation} 
Consider the innersum above. We divide the interval of summation $[d,N]$ into $\bigcup_{k=1}^{\left\lfloor \frac{N}{d}\right\rfloor-1}[kd, (k+1)d)\cup\left[\left\lfloor\frac{N}{d} \right\rfloor d,N\right]$. If $n\in [kd, (k+1)d)$, then $\left\lfloor\frac{n}{d} \right\rfloor = k$. So (\ref{parsumandlimstar}) becomes
\begin{align}\label{relthm2new}
\sum_{d\leq N}\mu\left(d\right)&\left(\sum_{k=1}^{\left\lfloor \frac{N}{d}\right\rfloor-1}\sum_{kd\leq n<(k+1)d}2^{\left\lfloor \frac{n}{d}\right\rfloor}+\sum_{\left\lfloor \frac{N}{d}\right\rfloor d\leq n\leq N}2^{\left\lfloor \frac{n}{d}\right\rfloor}\right)+O(N^2)\notag\\
&= \sum_{d\leq N}\mu\left(d\right)\left(d\sum_{k=1}^{\left\lfloor \frac{N}{d}\right\rfloor-1}2^k+2^{\left\lfloor \frac{N}{d}\right\rfloor}\left(N-d\left\lfloor \frac{N}{d}\right\rfloor +1\right)\right)+O(N^2)\notag\\
&= \sum_{d\leq N}d\mu\left(d\right)2^{\left\lfloor \frac{N}{d}\right\rfloor}+\sum_{d\leq N}\mu\left(d\right)2^{\left\lfloor \frac{N}{d}\right\rfloor}\left(N-d\left\lfloor \frac{N}{d}\right\rfloor +1\right)+O\left(N^2\right)
\end{align}
We see from (\ref{relthm2new}) that the main terms can be obtained from the small value of $d$. Expanding the sum for $d=1, 2, 3, 4$, we obtain
\begin{align}\label{relthm2new2}
2^{N+1}-2^{\left\lfloor \frac{N}{2}\right\rfloor}\left(N-2\left\lfloor \frac{N}{2}\right\rfloor+3\right)-2^{\left\lfloor \frac{N}{3}\right\rfloor}\left(N-3\left\lfloor \frac{N}{3}\right\rfloor+4\right)+O\left(\sum_{5\leq d\leq N}d2^{\left\lfloor \frac{N}{d}\right\rfloor}\right)
\end{align}
We have 
\begin{equation}\label{relthm2new3}
\sum_{5\leq d\leq N}d2^{\left\lfloor \frac{N}{d}\right\rfloor} \ll 2^{\left\lfloor \frac{N}{5}\right\rfloor}+\sum_{6\leq d\leq N}N2^{\left\lfloor \frac{N}{6}\right\rfloor}\ll 2^{\frac{N}{5}}
\end{equation}
We obtain (i) from (\ref{relthm2new}), (\ref{relthm2new2}), and (\ref{relthm2new3}).

Applying Lemma \ref{AKnew}(i), and \ref{NaAK}(i), we obtain
\begin{align*}
\sum_{n\leq N}\Phi\left(n\right) &= 1+\sum_{n\leq N-1}\Phi\left(n+1\right)\\
&=1+2\sum_{n\leq N-1}\left(f\left(n+1\right)-f\left(n\right)\right)\\
&=2f(N)-1\\
&= 2\left(\sum_{d\leq N}\mu\left(d\right)\left(2^{\left\lfloor \frac{N}{d}\right\rfloor}-1\right)\right)-1
\end{align*}
Similar to the proof of (i), we expand the sum for $d=1, 2, 3, 4$ to obtain (ii). Next we write, 
$$
\sum_{n\leq N}D\left(n\right) =\sum_{n\leq N}\sum_{d\mid n}f\left(d\right) = \sum_{dk\leq N}f\left(d\right) = \sum_{k\leq N}\sum_{d\leq\frac{N}{k}}f\left(d\right).
$$
Recall that $\left\lfloor \frac{\left\lfloor x\right\rfloor}{n}\right\rfloor = \left\lfloor \frac{x}{n}\right\rfloor$ for every $x\in \mathbb R$. Applying (i) to the above sum, we get
\begin{align}\label{relset2thm5eq1}
\sum_{n\leq N}D\left(n\right) = \sum_{k\leq N}2^{\left\lfloor \frac{N}{k}\right\rfloor+1}-2^{\left\lfloor \frac{N}{2k}\right\rfloor}\left(\left\lfloor \frac{N}{k}\right\rfloor-2\left\lfloor \frac{N}{2k}\right\rfloor+3\right)+O\left(N2^{\frac{N}{3}}\right) 
\end{align}
Now $\displaystyle \sum_{3\leq k\leq N}2^{\left\lfloor \frac{N}{k}\right\rfloor+1}-2^{\left\lfloor \frac{N}{2k}\right\rfloor}$$\left(\left\lfloor \frac{N}{k}\right\rfloor-2\left\lfloor \frac{N}{2k}\right\rfloor+3\right)\ll$$\displaystyle \sum_{k\leq N}2^{\frac{N}{3}}\ll $$N2^{\frac{N}{3}}$. 

So (\ref{relset2thm5eq1}) becomes
\begin{align*}
\sum_{n\leq N}D\left(n\right) &= 2^{N+1}-2^{\left\lfloor \frac{N}{2}\right\rfloor}\left(N-2\left\lfloor \frac{N}{2}\right\rfloor+3\right)+2^{\left\lfloor \frac{N}{2}\right\rfloor+1}+O\left(N2^{\frac{N}{3}}\right)\\
&= 2^{N+1}-2^{\left\lfloor \frac{N}{2}\right\rfloor}\left(N-2\left\lfloor \frac{N}{2}\right\rfloor+1\right)+O\left(N2^{\frac{N}{3}}\right).
\end{align*}
This completes the proof.
\end{proof}

\begin{corollary}\label{partialsumcor3}
We obtain the following:
\begin{itemize}
\item[(i)] $\displaystyle \lim_{\substack{N\rightarrow\infty\\N\;\text{odd}}} \frac{\left|\sum_{n\leq N}f\left(n\right)-2^{N+1}\right|}{2^{\left\lfloor \frac{N}{2}\right\rfloor}} = 4$,
\item[(ii)] $\displaystyle \lim_{\substack{N\rightarrow\infty\\N\;\text{even}}} \frac{\left|\sum_{n\leq N}f\left(n\right)-2^{N+1}\right|}{2^{\left\lfloor \frac{N}{2}\right\rfloor}} = 3$,
\item[(iii)] $\displaystyle \lim_{N\rightarrow \infty} \frac{\left|\sum_{n\leq N}\Phi\left(n\right)-2^{N+1}\right|}{2^{\left\lfloor \frac{N}{2}\right\rfloor}} = 2$,
\item[(iv)]  $\displaystyle \lim_{\substack{N\rightarrow \infty\\N\;\text{odd}}} \frac{\left|\sum_{n\leq N}D\left(n\right)-2^{N+1}\right|}{2^{\left\lfloor \frac{N}{2}\right\rfloor}} = 2$,\quad\text{and}  
\item[(v)] $\displaystyle \lim_{\substack{N\rightarrow \infty\\N\;\text{even}}} \frac{\left|\sum_{n\leq N}D\left(n\right)-2^{N+1}\right|}{2^{\left\lfloor \frac{N}{2}\right\rfloor}} = 1$.
\end{itemize}
\end{corollary}
\begin{proof}
By Theorem \ref{relpartialthm2}(i), we see that 
$$
\frac{\left|\sum_{n\leq N}f\left(n\right)-2^{N+1}\right|}{2^{\left\lfloor \frac{N}{2}\right\rfloor}} = N-2\left\lfloor \frac{N}{2}\right\rfloor+3+O\left(2^{\frac{N}{3}-\left\lfloor \frac{N}{2}\right\rfloor}\right).
$$
Note that $N-2\left\lfloor \frac{N}{2}\right\rfloor+3 = \begin{cases}
3\quad&\text{if $N$ is even};\\
4\quad&\text{if $N$ is odd},
\end{cases}$
and $2^{\frac{N}{3}-\left\lfloor \frac{N}{2}\right\rfloor}\rightarrow 0$ as $N\rightarrow \infty$. So we obtain (i) and (ii). Similarly, we can apply Theorem \ref{relpartialthm2}(ii) and \ref{relpartialthm2}(iii) to obtain (iii), (iv) and (v).
\end{proof}
\begin{corollary}\label{1to6}
The limits given in (\ref{relintrolimsupeq1}) to (\ref{relintrolimsupeq6}) hold. 
\end{corollary}
\begin{proof}
By Corollary \ref{partialsumcor3}(ii), we see that $\displaystyle \lim_{\substack{N\rightarrow\infty\\N\;\text{even}}} \frac{\left|\sum_{n\leq N}f\left(n\right)-2^{N+1}\right|}{2^{\frac{N}{2}}} = 3$, and by Corollary \ref{partialsumcor3}(i), we have 
\begin{align*}
\lim_{\substack{N\rightarrow\infty\\N\;\text{odd}}} \frac{\left|\sum_{n\leq N}f\left(n\right)-2^{N+1}\right|}{2^{\frac{N}{2}}} &= \lim_{\substack{N\rightarrow\infty\\N\;\text{odd}}} \frac{\left|\sum_{n\leq N}f\left(n\right)-2^{N+1}\right|}{2^{\frac{N-1}{2}}\sqrt 2} \\
&= \frac{1}{\sqrt 2}\lim_{\substack{N\rightarrow\infty\\N\;\text{odd}}} \frac{\left|\sum_{n\leq N}f\left(n\right)-2^{N+1}\right|}{2^{\left\lfloor \frac{N}{2}\right\rfloor}} \\
& = \frac{4}{\sqrt 2} = 2\sqrt 2.
\end{align*}
This gives (\ref{relintrolimsupeq1}) and (\ref{relintrolimsupeq2}). The proof of (\ref{relintrolimsupeq3}), (\ref{relintrolimsupeq4}), (\ref{relintrolimsupeq5}), and (\ref{relintrolimsupeq6}) is similar.
\end{proof}
We know from (\ref{sim}) that $f\left(n\right)$, $\Phi\left(n\right)$ and $D\left(n\right)$ are asymptotic to $2^n$. So we expect that $\sum_{n\leq N} \frac{f\left(n\right)}{2^n}$, $\sum_{n\leq N} \frac{\Phi\left(n\right)}{2^n}$, and $\sum_{n\leq N} \frac{D\left(n\right)}{2^n}$ are asymptotic to $N$. But this does not give much the information on the error terms $\left|\sum_{n\leq N} \frac{f\left(n\right)}{2^n}-N\right|$, $\left|\sum_{n\leq N} \frac{\Phi\left(n\right)}{2^n}-N\right|$, and $\left|\sum_{n\leq N} \frac{D\left(n\right)}{2^n}-N\right|$. We show in the next corollary that the error terms are small. 
\begin{corollary}
\begin{itemize}
\item [] 
\item [(i)] $\displaystyle \sum_{n\leq N}\frac{f\left(n\right)}{2^n} = N+1+\left(\log 2\right)\int_1^\infty\frac{\sum_{n\leq t}f\left(n\right)-2^{\left\lfloor t\right\rfloor+1}}{2^t}dt+O\left(2^{-\frac{N}{2}}\right)$.
\item [(ii)]$\displaystyle \sum_{n\leq N}\frac{\Phi\left(n\right)}{2^n} = N+1+\left(\log 2\right)\int_1^\infty\frac{\sum_{n\leq t}\Phi\left(n\right)-2^{\left\lfloor t\right\rfloor+1}}{2^t}dt+O\left(2^{-\frac{N}{2}}\right)$.
\item [(iii)] $\displaystyle \sum_{n\leq N}\frac{D\left(n\right)}{2^n} = N+1+\left(\log 2\right)\int_1^\infty\frac{\sum_{n\leq t}D\left(n\right)-2^{\left\lfloor t\right\rfloor+1}}{2^t}dt +O\left(2^{-\frac{N}{2}}\right)$.
\end{itemize}
\end{corollary}
\begin{proof}
Let $\displaystyle F\left(t\right)=\sum_{n\leq t}f\left(n\right)$. Then by partial summation (see for example in \cite[p.\ 77]{Ap}  or \cite[ p.\ 488]{MV}), we see that
\begin{equation}\label{relset2eq1}
\sum_{n\leq N}\frac{f\left(n\right)}{2^n} = \frac{F\left(N\right)}{2^N}+(\log 2)\int_1^N\frac{F(t)}{2^t}dt
\end{equation}
By Theorem \ref{relpartialthm2}(i), for $t\geq 1$, we can write $F\left(t\right) = 2^{\left\lfloor t\right\rfloor+1}+g\left(t\right)$ where $g\left(t\right) = O\left(2^{\frac{t}{2}}\right)$. Then (\ref{relset2eq1}) becomes
\begin{equation}\label{relset2eq2}
\sum_{n\leq N}\frac{f\left(n\right)}{2^n} = 2+\left(\log 2\right)\int_1^N\frac{2^{\left\lfloor t\right\rfloor+1}}{2^t}+\frac{g(t)}{2^t}dt+O\left(2^{-\frac{N}{2}}\right)
\end{equation}
Consider 
\begin{align}\label{relset2eq3}
\int_1^N\frac{2^{\left\lfloor t\right\rfloor+1}}{2^t}dt &= \sum_{k=1}^{N-1}\int_k^{k+1}\frac{2^{\left\lfloor t\right\rfloor+1}}{2^t}dt\notag\\
&= \sum_{k=1}^{N-1}\int_k^{k+1}\frac{2^{k+1}}{2^t}dt\notag\\
&= \sum_{k=1}^{N-1}2^{k+1}\left[\frac{-2^{-t}}{\log 2}\right]_k^{k+1} = \frac{N-1}{\log 2}
\end{align}
Since $g(t) = O\left(2^{\frac{t}{2}}\right)$, $\int_1^\infty\frac{g(t)}{2^t}dt$ converges and $\int_N^\infty\frac{g(t)}{2^t}dt \ll \int_N^\infty2^{-\frac{t}{2}}dt \ll 2^{-\frac{N}{2}}$. So
\begin{equation}\label{relset2eq4}
\int_1^N\frac{g(t)}{2^t}dt = \int_1^\infty\frac{g(t)}{2^t}dt+O\left(2^{-\frac{N}{2}}\right)
\end{equation}
From (\ref{relset2eq2}), (\ref{relset2eq3}) and (\ref{relset2eq4}), we obtain 
\begin{align*}
\sum_{n\leq N}\frac{f\left(n\right)}{2^n} &= 2+(\log 2)\left(\frac{N-1}{\log 2}+\int_1^\infty\frac{g(t)}{2^t}dt\right)+O\left(2^{-\frac{N}{2}}\right)\\
&= N+1+(\log 2)\int_1^\infty\frac{g(t)}{2^t}dt+O\left(2^{-\frac{N}{2}}\right).
\end{align*}
The proof of (ii) and (iii) is similar.
\end{proof}
We investigate some combinatorial properties of $D\left(n\right)$ in the next section.
\section{Combinatorial properties}
We will give a combinatorial interpretation of $D\left(n\right)$. But it may be useful later to do it in a more general setting. So we introduce the following definition. Throughout, let $X$, $X_d$, and $\frac{1}{d}X$ denote, respectively, a nonempty finite set of positive integers, $\left\{x\in X\;:\;d\mid x\right\}$ and $\left\{\frac{x}{d}\;:\:x\in X\right\}$. 
\begin{definition}
Let $D\left(X,n\right)$ denote the number of nonempty subsets $A$ of $X$ such that $\gcd\left(A\right)\mid n$, and let $f\left(X\right)$ denote the number of relatively prime subsets of $X$.
\end{definition}
\begin{theorem}\label{combithm6}
Let $X$ be a nonempty finite set of positive integers. Then 
$$
D\left(X,n\right) = \sum_{d\mid n}f\left(\frac{1}{d}X_d\right).
$$
\end{theorem}
\begin{proof}
 We begin with 
\begin{equation}\label{thm1eq1}
D\left(X,n\right) = \sum_{\substack{\emptyset \neq A \subseteq X\\\gcd \left(A\right)\mid n}} 1 = \sum_{d\mid n}\sum_{\substack{\emptyset \neq A \subseteq X\\\gcd \left(A\right) = d}}1
\end{equation}
The condition $\gcd \left(A\right)=d$ means that $d$ divides all elements of $A$ and $\gcd\left(\frac{1}{d}A\right)=1$. So $\emptyset\neq A\subseteq X$ and $\gcd\left(A\right)=d$ if and only if $\emptyset\neq A\subseteq X_d$ and $\gcd\left(\frac{1}{d}A\right)=1$. Therefore the innersum in (\ref{thm1eq1}) is equal to
$$
\sum_{\substack{\emptyset \neq A \subseteq X_d\\\gcd\left(\frac{1}{d}A\right)=1}} 1 = \sum_{\substack{\emptyset \neq \frac{1}{d}A \subseteq \frac{1}{d}X_d\\\gcd\left(\frac{1}{d}A\right)=1}} 1 = \sum_{\substack{\emptyset \neq B \subseteq \frac{1}{d}X_d\\\gcd\left(B\right)=1}} 1 = f\left(\frac{1}{d}X_d\right).
$$ 
Hence 
$$
D\left(X,n\right) = \sum_{d\mid n}f\left(\frac{1}{d}X_d\right).
$$
\end{proof}
\begin{corollary}
$D\left(n\right)$ is equal to the number of subsets $A$ of $\{1,2,\ldots,n\}$ such that $\gcd\left(A\right)\mid n$. In other words, $D\left(n\right) = D\left(\left\{1,2,\ldots,n\right\},n\right)$.
\end{corollary}
\begin{proof}
Let $X = \left\{1,2,\ldots,n\right\}$. Then $X_d = \left\{d,2d,\ldots,\left\lfloor \frac{n}{d}\right\rfloor d\right\}$. Therefore $\frac{1}{d}X_d = \left\{1,2,\ldots,\left\lfloor \frac{n}{d}\right\rfloor\right\}$. By the definition, we see that 
$$
f\left(\frac{1}{d}X_d \right)= f\left(\left\{1,2,\ldots,\left\lfloor \frac{n}{d}\right\rfloor\right\}\right)= f\left(\left\lfloor \frac{n}{d}\right\rfloor\right).
$$
Then by Theorem \ref{combithm6}, we see that 
$$
D\left(X,n\right) = \sum_{d\mid n}f\left(\left\lfloor \frac{n}{d}\right\rfloor\right) = \sum_{d\mid n}f\left(d\right) = D\left(n\right).
$$
Therefore $D\left(n\right)$ is equal to the number of subsets $A$ of $\{1,2,\ldots,n\}$ such that $\gcd\left(A\right)\mid n$.
\end{proof}
\begin{theorem}
Let $d\left(n\right)$ be the number of positive divisors of $n$. Then $D\left(n\right)+d\left(n\right)+1\equiv 0\pmod 3$ for every $n\geq 1$.
\end{theorem}
\begin{proof}
By Lemma \ref{AKnew}(i) and \ref{AKnew}(ii), we see that 
$$
f(n+1) \equiv f\left(n\right)\pmod 3\quad\text{for every $n\geq 2$}.
$$
This implies that $f\left(n\right)\equiv f(2)\equiv 2\pmod 3$ for every $n\geq 2$. Then
$$
D\left(n\right) = \sum_{d\mid n}f\left(d\right) = f(1)+\sum_{\substack{d\mid n\\d\geq 2}}f\left(d\right)\equiv 1+2(d\left(n\right)-1)\pmod3.
$$
This implies that $D\left(n\right)+d\left(n\right)+1 \equiv 0\pmod 3$.
\end{proof}

\textbf{Comments and Open Questions}
\begin{itemize}
	\item [1)] There is a small miscalculation in the formulas for $\Phi\left(n\right)$ and its generalizations in the literature. The right one is $\Phi\left(n\right) = \sum_{d\mid n}\mu\left(d\right)\left(2^{\frac{n}{d}}-1\right)$ (Lemma \ref{NaAK}(ii)) which corresponds to A038199 in Sloane's On-Line Encyclopedia of Integer Sequences \cite{Sl}. The wrong one is $\Phi\left(n\right) = \sum_{d\mid n}\mu\left(d\right)2^{\frac{n}{d}}$ which is usually referred to as A027375. At this moment (June 10, 2012) the sequence $\Phi\left(n\right)$ is put in the wrong place at A027375. The author will notify this to Professor Sloane or the database manager when this article is ready for publication. Fortunately, there is little danger since both sequences coincide for all $n \geq 2$. This is because we have the well known identity 
	$$
	\sum_{d\mid n}\mu\left(d\right) = \begin{cases}
	1\quad&\text{if $n=1$};\\
	0\quad&\text{if $n>1$}.
	\end{cases}
	$$
	\item [2)] The sequence $D\left(n\right)$ is new and will be submitted to Sloane's On-Line Encyclopedia of Integer Sequences \cite{Sl} soon.
	\item [3)] As suggested by the limits given in (\ref{relintrolimsupeq3}) to (\ref{relintrolimsupeq6}), on average, the sequence $D\left(n\right)$ lies closer to $2^n$ than $\Phi\left(n\right)$. But for certain $n$, $\Phi\left(n\right)$ may lie closer to $2^n$ than $D\left(n\right)$. Considering Table \ref{tablerela1} more carefully, we see that $\Phi\left(n\right)$ lie closer to $2^n$ for all odd $n$ from 5 to 15. Therefore 
	\begin{equation}\label{commenteq}
	\text{the sign of $D\left(n\right)-\Phi\left(n\right)$ is alternating for $4\leq n\leq 15$.}
	\end{equation}
	So natural question arises:
	\begin{itemize}
		\item [3.1] Does (\ref{commenteq}) holds for all $n\geq 4$? We check that (\ref{commenteq}) holds for $4\leq n\leq 30$. But we do not have a proof for $n\geq 31$. It is possible that (\ref{commenteq}) does not hold for some $n\geq 4$. In this case, we may ask a weaker question: 
		\item [3.2] Does $D\left(n\right)-\Phi\left(n\right)$ change sign infinitely often?
		
		Other possible research questions are the following:
		\item [3.3] Can we say something about $\limsup_{n\to\infty}\frac{2^n-D\left(n\right)}{2^n-\Phi\left(n\right)}$, $\liminf_{n\to\infty}\frac{2^n-D\left(n\right)}{2^n-\Phi\left(n\right)}$, $\sum_{n\leq N}\frac{2^n-D\left(n\right)}{2^n-\Phi\left(n\right)}$, or $\sum_{n\leq N}\frac{2^n-\Phi\left(n\right)}{2^n-D\left(n\right)}$?
		\item [3.4] Are $D\left(n\right)$ and $\Phi\left(n\right)$ a perfect power for some $n\geq 2$? (Ayad and Kihel \cite{AK} prove that $f\left(n\right)$ is never a square for $n\geq 2$).
		\item [3.5] Are the sequences $D\left(n\right)$ and $\Phi\left(n\right)$ periodic modulo a prime $p$? (Ayad and Kihel \cite{AK} show that the sequence $f\left(n\right)$ is not periodic modulo $p$ for any $p\neq 3$).
	\end{itemize}
\end{itemize}
\section{Acknowledgement}
The author recieves financial support from Faculty of Science, Silpakorn University, Thailand, contract number RGP 2555-07.

\end{document}